\documentclass[12pt, reqno]{amsart}
\usepackage[top=3cm, bottom=3.5cm, left=3cm, right=3cm]{geometry}
\usepackage{amsmath,amsfonts,amssymb,amsthm}
\usepackage{mathtools}
\usepackage[utf8]{inputenc}
\usepackage{hyperref}
\usepackage{mathrsfs}
\usepackage{tikz-cd}

\newcommand{\qbar}{{\bar{q}}}
\newcommand{\jbar}{{\bar{j}}}

\newcommand{\lbar}{{\bar{l}}}

\def\ddb{\partial\bar\partial}
\def\tr{\operatorname{tr}}
\def\Re{\operatorname{Re}}
\def\Im{\operatorname{Im}}
\def\Ric{\operatorname{Ric}}

\usepackage{geometry}
\usepackage{booktabs}
\usepackage{array}
\usepackage{paralist}
\usepackage{verbatim}
\usepackage{subfig}
\usepackage{fancyhdr}
\author{Shuang Liang}
\address{Department of Mathematics\\
  Columbia University\\
  New York, NY USA 10027}
\email[S. Liang]{sliang@math.columbia.edu}
\author{Xi Sisi Shen}
\address{Department of Mathematics\\
  Columbia University\\
  New York, NY USA 10027}
\email[X. S. Shen]{xss@math.columbia.edu}
\author{Kevin Smith}
\address{Department of Mathematics\\
  Columbia University\\
  New York, NY USA 10027}
\email[K. Smith]{kjs@math.columbia.edu}

\newtheorem{proposition}{Proposition}
\newtheorem{thm}{Theorem}
\newtheorem{corollary}{Corollary}

\newtheorem{remark}{Remark}
\newtheorem{definition}{Definition}
\numberwithin{equation}{section}

\begin{document}
\bibliographystyle{amsplain}

\title{The continuity equation for Hermitian metrics: Calabi estimates, Chern scalar curvature and Oeljeklaus-Toma manifolds}

\begin{abstract}
We prove local Calabi and higher order estimates for solutions to the continuity equation introduced by La Nave-Tian and extended to Hermitian metrics by Sherman-Weinkove. We apply the estimates to show that on a compact complex manifold the Chern scalar curvature of a solution must blow up at a finite-time singularity. Additionally, starting from certain classes of initial data on Oeljeklaus-Toma manifolds we prove Gromov-Hausdorff and smooth convergence of the metric to a particular non-negative $(1,1)$-form as $t\to\infty$.
\end{abstract}
\maketitle

\section{Introduction}\label{Intro}

\noindent Calabi in 1958 \cite{calabi58} first proved third order estimates for the complex Monge-Amp\`ere equation. This type of estimate was later used by Aubin \cite{aubin70} and Yau \cite{yau78} as a part of their proofs for showing the existence of K\"ahler-Einstein metrics on manifolds with $c_1(X)<0$, and on manifolds with $c_1(X)<0$ or $c_1(X)=0$, respectively. Since then, these techniques have been applied to numerous settings as surveyed in \cite{pss12}. In 1982, Hamilton introduced the Ricci flow \cite{hamilton82} which has proven to be very successful for studying the geometric and topological properties of manifolds in a variety of settings, e.g. by Hamilton in \cite{hamilton95}, by Brendle-Schoen in \cite{bs08} and Perelman in \cite{perelman1,perelman2,perelman3} . If the initial data is K\"ahler, then the Ricci flow preserves the K\"ahler condition and is referred to as the K\"ahler-Ricci flow with evolution equation given by
\begin{equation}\label{Intro:KRF}
\partial_t \omega = -\Ric\omega .
\end{equation}
The solutions to the K\"ahler-Ricci flow solve a parabolic complex Monge-Amp\`ere equation and Cao \cite{cao85} used the K\"ahler-Ricci flow to provide a parabolic proof of the existence of K\"ahler-Einstein metrics on manifolds with $c_1(X)<0$ and $c_1(X)=0.$ 
Another application of the K\"ahler-Ricci flow is the \textit{Analytic Minimal Model Program} proposed by Song-Tian \cite{st07, st12, st17} and Tian \cite{tian08} to carry out Mori's minimal model program \cite{km98} using the flow. The goal of the program is to classify algebraic varieties by running the K\"ahler-Ricci flow to reduce an algebraic variety to the minimal model in its birational class. For instance, the K\"ahler-Ricci flow can be set up to have the behavior of ``blowing down" $(-1)$-curves on a complex surface, a process that when finitely iterated will lead to the minimal model. 

If the initial data is Hermitian but not necessarily K\"ahler and $\Ric\omega$ is interpreted as the Ricci curvature of the Chern connection, then Equation \eqref{Intro:KRF} defines the Chern-Ricci flow introduced by Gill \cite{gill11}. There is a plethora of literature on the study of non-K\"ahler flows \cite{acs,at21,bv17, bv20, bx, dang22,edwards21, ftwz16,fppz21, fppz22,  kawamura19, liu-yang12, ppz18-1, ppz18-2, shen22, streets_tian12,tw13, tw15,tw22, twy15,yury}.
 
La Nave-Tian \cite{nt15} introduced the continuity equation as an elliptic alternative to the K\"ahler-Ricci flow for carrying out the Analytic Minimal Model Program. Starting from a K\"ahler metric $\omega_0$, the continuity equation is given by
\begin{equation}\label{Intro:ContinuityEquation}
    \omega = \omega_0 - t\Ric\omega.
\end{equation}
This equation is similar in nature to the K\"ahler-Ricci flow, with the tangent slope $\partial_t\omega$ being replaced by the secant slope $(\omega-\omega_0)/t$. The continuity equation was extended to Hermitian metrics by Sherman-Weinkove in \cite{sw20} where they prove that the maximal existence time
\begin{equation}\label{Intro:MaximalExistenceTime}
    T = \sup\{t > 0 : \exists\psi\in C^\infty(X) \text{ with } \omega_0 - t\Ric\omega_0 + i\ddb\psi > 0\}
\end{equation}
for a solution of Equation \eqref{Intro:ContinuityEquation} agrees with the maximal existence time of the Chern-Ricci flow \cite{tw15}. They also show that a solution to the continuity equation starting at a Gauduchon metric on a minimal elliptic surface converges to a metric on the base curve, and that the respective manifolds converge in the Gromov-Hausdorff sense. In both the K\"ahler and Hermitian settings, the convergence behavior of the metric along the flow and the corresponding continuity equation are closely related, with the advantages being that solutions to the continuity equation are generally simpler to study: the maximum principle is more straightforward, $C^0$ estimates are often more immediate, and that a lower Ricci curvature bound always holds. With all of these analytic advantages, the continuity equation is a promising tool to use to carry out a Hermitian version of the Analytic Minimal Model Program.

We will also consider the normalized continuity equation
\begin{equation}\label{Intro:NormalizedContinuityEquation}
(t+1)\omega = \omega_0 - t\Ric\omega.
\end{equation}
Scaling by a factor of $t+1$ converts solutions of \eqref{Intro:NormalizedContinuityEquation} into solutions of \eqref{Intro:ContinuityEquation} and vice versa. In previous work by the second and third authors \cite{ss22} solutions to Equations \eqref{Intro:ContinuityEquation} and \eqref{Intro:NormalizedContinuityEquation} were studied on Hopf and Inoue surfaces.

In this paper, we first prove local Calabi estimates for solutions to the continuity equation:
\begin{thm}\label{Intro:ThmC3}
Let $\chi(t)$ solve the continuity equation \eqref{Intro:ContinuityEquation} or the normalized continuity equation \eqref{Intro:NormalizedContinuityEquation} starting at a Hermitian metric $\chi_0$ in a neighborhood of $B_r$ for fixed $r\in (0,1)$ and $t\in [0,T]$. Assume that there exists a Hermitian metric $\hat\chi$ and a uniform constant $K > 1$ such that for all $t\in [0,T],$
\begin{align}
    K^{-1}\hat\chi \le \chi(t)\le K\hat\chi \ \text{ on }B_r.\label{Intro:C3_assumption}
\end{align}
Then for any fixed $\varepsilon > 0$ there exists a uniform constant $C$ depending only on $K,\chi_0,\hat\chi,\varepsilon$ such that for all $t\in [\varepsilon,T]$
\begin{align}\label{Intro:C3Estimate}
    |\hat\nabla \chi|^2_\chi \le \frac{C}{r^2} \text{ on } B_{r/2}
\end{align}
where $\hat\nabla$ is the Chern connection of $\hat\chi.$
\end{thm}
\begin{remark}
    More specifically, the constant above will depend on $K$, $\varepsilon$, the first derivative of $\chi_0$, first and second derivatives of the torsion of $\chi_0$, and the curvature and first derivative of the curvature of $\hat\chi$.
\end{remark}

\noindent The specific dependence of the constants will be important when we apply the following Corollary to the rescaled metric in the setting of Section \ref{OT}. Note that the above estimate is independent of $T$ and so $T$ can be taken to be infinity. Once estimates up to third order are obtained, a straightforward bootstrapping argument can be used to obtain higher order estimates:

\begin{corollary}\label{Intro:CorCk}
Let $\chi$ solve \eqref{Intro:ContinuityEquation} or \eqref{Intro:NormalizedContinuityEquation} and assume that we already have the estimate \eqref{Intro:C3Estimate}. Then there exist constants $C_k,\gamma_k$ for $k=1,2,3\ldots$ such that
\begin{align*}
    |\hat\nabla^k_{\mathbb{R}}\chi|^2_{\hat\chi} \leq \frac{C_k}{r^{\gamma_k}} \text{ on } B_{r/4},
\end{align*}
where $\hat\nabla_{\mathbb{R}}$ is the Levi-Civita covariant derivative with respect to $\hat\chi.$
\end{corollary}

\noindent In \cite{sw13}, Sherman-Weinkove prove analogous higher order estimates along the Chern-Ricci flow. We also note that an alternative method for obtaining higher order estimates would be to use the method of Evans-Krylov \cite{evans82,krylov83} as was done in \cite{gs15}. The above theorem may be helpful for studying solutions to the continuity equation on compact subsets away from a subvariety, which is the case for the Chern-Ricci flow. We then apply these estimates to prove the blow-up of the Chern scalar curvature at a finite-time singularity:

\begin{thm}\label{Intro:ThmBlowup}
Let $X$ be a compact complex manifold with a Hermitian metric $\omega_0$, and let $\omega$ be the unique solution of either the continuity equation \eqref{Intro:ContinuityEquation} or the normalized continuity equation \eqref{Intro:NormalizedContinuityEquation} which exists until the maximal existence time $T$. Then either $T = \infty$ or the Chern scalar curvature $R$ of $\omega$ satisfies
\begin{equation}
    \limsup_{t\rightarrow T} \left(\sup_X R\right) = \infty.
\end{equation}
\end{thm}

\noindent Results of this type were first studied in \cite{zhang10} for the K\"ahler-Ricci flow and in \cite{gs15} for the Chern-Ricci flow and the proofs are by contradiction. Due to the simple nature of the continuity equation in comparison to the flows, the desired contradiction is reached more immediately in this setting. The above result was proved for the continuity equation in \cite{wondo22} in the case where $\omega_0$ is K\"ahler.

In the last part of this paper, we use the Calabi estimates in Theorem \ref{Intro:ThmC3} to obtain smooth convergence results on Inoue surfaces and Oeljeklaus-Toma (OT) manifolds. OT manifolds are higher-dimensional analogues of Inoue surfaces of type $S_M$ introduced by \cite{ot05} whose construction arises from number theory. They are non-K\"ahler compact complex solvmanifolds with negative Kodaira dimension admitting no Vaisman metrics \cite{kasuya13,ov11}. It was shown in \cite{bo15} that they do not admit any analytic hypersurfaces and have algebraic dimension zero. It was shown in \cite{otiman22} that OT manifolds do not admit any balanced metrics but always have locally conformally balanced metrics. See \cite{ados22,apv20,fz22,io19,ovv19,ov13,ot21,verbitsky2,verbitsky14,vuletescu14} for recent progress and open problems on OT manifolds. Using the local Calabi estimates of Theorem \ref{Intro:ThmC3}, alongside an adaptation of a rescaling argument used by Angella-Tosatti \cite{at21} in the setting of Inoue surfaces, we can prove higher order estimates and smooth convergence of solutions to the continuity equation starting from a certain class of initial data. 

\newpage
\begin{thm}\label{Intro:ThmOT}
Let $X$ be an Oeljeklaus-Toma manifold of $\dim_{\mathbb C}X = n$ with the Hermitian metric $\omega_{OT}$, let $\omega$ be the unique solution on $X \times [0,T)$ of \eqref{Intro:NormalizedContinuityEquation} with $\omega_{0} = \omega_{OT} + i\ddb\varphi_0$, where $\varphi_0$ is smooth ($\omega_{OT}$ to be defined in Section \ref{OT}). Then we have $T = \infty$ and
\begin{equation}\label{Intro:eqnOTC2}
    \omega \xrightarrow{C^0} -\Ric(\omega_{OT}).
\end{equation}
In particular, $(X,\omega)$ converges to $(\mathbb{T}^{n-1},d)$ in the Gromov-Hausdorff sense as $t\to\infty$, where $d$ is a flat Riemannian metric on $\mathbb{T}^{n-1}$ determined by $X$. Additionally, there exists a uniform constant $C$ such that for $1 \leq t \leq \infty$
\begin{equation}
    -C\omega\leq\Ric\omega\leq C\omega.
\end{equation}
If we further assume that $\sqrt{-1}\ddb \varphi_0$ is weakly parallel along the leaves (defined in Definition \ref{weaklyparallel}), then the convergence \eqref{Intro:eqnOTC2} is smooth.
\end{thm}

\noindent For the Chern-Ricci flow, Gromov-Hausdorff convergence on OT manifolds \cite{zheng15} and smooth convergence on Inoue surfaces \cite{at21} are known to hold. However, the Ricci curvature bounds are not known in that setting. In addition, the above theorem strengthens the convergence on Inoue surfaces for the continuity equation obtained in previous work of the second and third authors \cite{ss22} to smooth convergence. Due to the nature of the continuity equation and its explicit dependence on the initial metric, we need to impose a further condition on the initial potential in order to obtain smooth convergence that would not be needed in the parabolic setting. 

This paper is organized as follows. In Section \ref{Prelim}, we establish the notation for the paper and provide some identities that will be used in the later sections. In Section \ref{C3}, we prove the local Calabi estimates of Theorem \ref{Intro:ThmC3}. Using these estimates, we prove Theorem \ref{Intro:ThmBlowup} in Section \ref{Blowup}; that the scalar curvature must blow up at a finite-time singularity. Finally, in Section \ref{OT}, we provide a proof of Theorem \ref{Intro:ThmOT} showing smooth convergence of solutions of the continuity equation on Oeljeklaus-Toma manifolds to a non-negative $(1,1)$-form and Gromov-Hausdorff convergence to $\mathbb{T}^{n-1}$. 

\section{Preliminaries}\label{Prelim}

\noindent In this section, we will establish the notation used in the paper and cover some identities that will be needed in subsequent sections. Given a Hermitian metric
\begin{align*}
    \omega = g_{\bar kj}\,idz^j\wedge d\bar z^k
\end{align*} its Chern connection $\nabla = \partial + \Gamma$ is defined by
\begin{align*}
    \Gamma^p_{jr} & = g^{p\bar q}\partial_jg_{\bar qr} \\
    \Gamma^p_{\bar kr} & = 0.
\end{align*}
We define the torsion
\begin{align*}
    T^p_{jr} = \Gamma^p_{jr} - \Gamma^p_{rj}
\end{align*}
the torsion $1$-form
\begin{align*}
\tau_j = T^p_{pj}
\end{align*}
and the curvature
\begin{align*}
    R_{\bar kj}{}^p{}_r = -\partial_{\bar k}\Gamma^p_{jr}.
\end{align*}
The first and second Chern Ricci curvatures are respectively
\begin{align*}
    R_{\bar kj} & = R_{\bar kj}{}^p{}_p \\
    R'_{\bar kj} & = R^p{}_{p\bar kj}.
\end{align*}
We also have the following well-known identities
\begin{align*}
	R_{\bar j i \bar l k} - R_{\bar jk \bar l i} &= - \nabla_{\bar j} T_{i k \bar l} \\
	R_{\bar j i\bar l k} - R_{\bar l i \bar j k} &= -\nabla_i T_{\jbar  \lbar k} \\
	R_{\bar j i \bar l k} - R_{\bar l k \bar j i} 
		&= - \nabla_{\bar j} T_{i k \bar l} - \nabla_k T_{\bar j  \bar l i} 
		= - \nabla_i T_{\bar j \bar l k} - \nabla_{\bar l} T_{i k \bar j} \\
	\nabla_p R_{\bar j i \bar l k} - \nabla_i R_{\bar j p \bar l k} 
		&= - T_{p i}{}^r R_{\bar j r \bar l k} \\
	\nabla_{\bar q} R_{\bar j i \bar l k} - \nabla_{\bar j} R_{\bar q i \bar l k} &= - T_{\bar q \bar j}{}^{\bar s} R_{\bar s i \bar l k}.
\end{align*}
We define the trace of a real $(1,1)$-form $\alpha = \alpha_{\bar kj}\,idz^j\wedge d\bar z^k$ with respect to $\omega$ by
\begin{align*}
    \tr_\omega \alpha = g^{j\bar k}\alpha_{\bar kj}.
\end{align*}
If $\chi = \chi_{\bar kj}\,idz^j\wedge d\bar z^k$ is another Hermitian metric with $d(\omega - \chi) = 0$, by a computation due to Cherrier \cite{cherrier87}
\begin{align}\label{Prelim:LaplacianTrace}
\Delta_\omega\tr_\chi\omega
& = -\chi^{p\bar q}R_{\bar qp}(\omega) + g^{j\bar k}R_{\bar kj}{}^{p\bar q}g_{\bar qp} \\
& \ \ \ \, + g^{j\bar k}R^p{}_{p\bar kj} - g^{j\bar k}R_{\bar kj}{}^p{}_p - g^{j\bar k}\chi^{p\bar q}\chi_{\bar sr}T^r_{pj}\bar T^s_{qk} \nonumber \\
& \ \ \ \, + g^{j\bar k}\chi^{p\bar q}g_{\bar sr}\Phi^r_{pj}\bar\Phi^s_{qk} \nonumber
\end{align}
and 
\begin{align}\label{Prelim:LaplacianLogTrace}
\Delta_\omega\log\tr_\chi\omega
& = \frac{1}{\tr_\chi\omega}\bigg\{-\chi^{p\bar q}R_{\bar qp}(\omega) + g^{j\bar k}R_{\bar kj}{}^{p\bar q}g_{\bar qp} \\
& \ \ \ \ \ \ \ \ \ \ \ \ \, + g^{j\bar k}R^p{}_{p\bar kj} - g^{j\bar k}R_{\bar kj}{}^p{}_p - g^{j\bar k}\chi^{p\bar q}\chi_{\bar sr}T^r_{pj}\bar T^s_{qk} \nonumber \\
& \ \ \ \ \ \ \ \ \ \ \ \ \, + 2\Re g^{j\bar k}\frac{\partial_j\tr_\chi\omega}{\tr_\chi\omega}\bar \tau_k + \frac{1}{\tr_\chi\omega}g^{j\bar k}\tau_j\bar \tau_k \nonumber \\
& \ \ \ \ \ \ \ \ \ \ \ \ \, + g^{j\bar k}\chi^{p\bar q}g_{\bar sr}\Phi^r_{pj}\bar\Phi^s_{qk} - \frac{1}{\tr_\chi\omega}g^{j\bar k}\Psi_j\bar\Psi_k\bigg\} \nonumber
\end{align}
where $R$ and $T$ are the curvature and torsion of $\chi$, and $\Phi^r_{pj} = g^{r\bar s} \nabla_pg_{\bar sj} + T^r_{pj}$ and $\Psi_j = \partial_j\tr_\chi\omega + \tau_j$. We have in \eqref{Prelim:LaplacianTrace}
\begin{equation}
g^{j\bar k}\chi^{p\bar q}g_{\bar sr}\Phi^r_{pj}\bar\Phi^s_{qk} \geq 0
\end{equation}
and by the methods used in the proof of \cite[Theorem 2.1]{tw10-2} (see also \cite[Theorem 3]{smith20}) we have in \eqref{Prelim:LaplacianLogTrace}
\begin{equation}\label{Prelim:AubinYau}
g^{j\bar k}\chi^{p\bar q}g_{\bar sr}\Phi^r_{pj}\bar\Phi^s_{qk} - \frac{1}{\tr_\chi\omega}g^{j\bar k}\Psi_j\bar\Psi_k \geq 0.
\end{equation}
From the continuity equation \eqref{Intro:ContinuityEquation} and \eqref{Intro:NormalizedContinuityEquation}, it is immediate that $\omega$ and $\omega_0$ differ by $\Ric\omega$ which is closed. From this, it follows that 
\begin{align}
    T_{ij\bar{k}}=(T_0)_{ij\bar{k}},\label{Prelim:Torsion}
\end{align}
where $T_{ij\bar{k}}=g_{\bar{k}r}T^r_{ij}$ and so torsion terms of the solution metric can be expressed in terms of torsion terms of the initial metric. We note that throughout this paper, the constant $C$ may change from line to line.

\section{Local Calabi estimates}\label{C3}

\subsection{$C^3$ estimate}\label{C3:C3}

In this subsection, we will prove Theorem \ref{Intro:ThmC3}. The proof follows in a similar way to Section 3 of \cite{sw13} except that we are working with an elliptic equation instead of a parabolic equation.

\begin{proof}
Let $\Delta$ and $\nabla$ be the Laplacian and covariant derivative with respect to $\chi$ and let $\hat\nabla$ be the covariant derivative with respect to $\hat\chi.$ It is important that $\chi$ and $\hat\chi$ differ by a closed form to ensure that we can control torsion and derivatives of torsion terms of $\chi$ in terms of those of $\hat\chi$ using Equation \eqref{Prelim:Torsion}. We prove Theorem \ref{Intro:ThmC3} by applying a maximum principle argument to the Calabi-type quantity \cite{calabi58,yau78} (see \cite{pss07} in the parabolic K\"ahler case)
\begin{align*}
	S := | \Psi |^2 = | \hat \nabla \chi |^2,
\end{align*}
where the norm is taken with respect to $\chi$. 
Firstly, we compute (see Section 3 of \cite{sw13} for details)
\begin{align}
    \label{DeltaS}
    \begin{split}
	\Delta S ={} & |\overline\nabla \Psi|^2 + |\nabla \Psi|^2 
		+ 2 \Re \left( 	( \Delta \Psi_{ij}{}^k ) \Psi^{ij}{}_k  \right) \\
		&+ (R_p{}^p{}_r{}^i \Psi^{rj}{}_k  
			+ R_p{}^p{}_r{}^j \Psi^{ir}{}_k  
			- R_p{}^p{}_k{}^r \Psi^{ij}{}_r  ) \Psi_{ij}{}^k .
    \end{split}
\end{align}
For the third term, we can use the identity from Equation (3.2) of \cite{sw13}
\begin{align*}
	\Delta \Psi_{ij}{}^k = - \nabla^{\bar q} R_{ \bar q i}{}^k {}_j
		+ \nabla^{\bar q} \hat{R}_{\bar q i}{}^k{}_j.
\label{C3:LaplacianPsi}
\end{align*}
Using the fact that $\chi(t)$ solves the continuity equation \eqref{Intro:ContinuityEquation} in a neighborhood of $B_r$, we have the following 
\begin{equation}\label{Ricci_identity}
R^k{}_j = -\frac{\chi^{k\bar{q}}\chi_{q\bar{j}}-\chi^{k\bar{q}}(\chi_0)_{\bar{q}j}}{t}, 
\end{equation}
\begin{align*}
\nabla_i R^k{}_j &= \frac{1}{t} \chi^{k\bar{q}}\nabla_i (\chi_0)_{\bar{q}j}\\
&= \frac{1}{t}(\chi^{k\bar{q}}\hat\nabla_i (\chi_0)_{\bar{q}j}-\chi^{k\bar{q}}(\chi_0)_{\bar{q}r}\Psi^r_{ij}).
\end{align*}
For the normalized continuity equation $\eqref{Intro:NormalizedContinuityEquation}$, we have a coefficient of $t+1$ in front of the $\chi$ terms which leads to the same bounds.
Now we can bound the term $\Delta\Psi_{ij}{}^k$ from below as follows
\begin{align*}
	\Delta \Psi_{ij}{}^k 
		=& 	- \nabla^\qbar R_{\qbar i}{}^k{}_j
			+ \nabla^\qbar \hat{R}_{\qbar i}{}^k{}_j \\
		=& \left( - \nabla_i R^k{}_j{}^p{}_p + T_{pi}{}^r R^p{}_r{}^k{}_j 
			+ \nabla_i \nabla_j T^{p k}{}_p + \nabla_i \nabla^p T_{pj}{}^k \right)
			 + \nabla^\qbar \hat{R}_{\qbar i}{}^k{}_j \\
		=&  -\frac{1}{t}(\chi^{k\bar{q}}\hat\nabla_i (\chi_0)_{\bar{q}j}-\chi^{k\bar{q}}(\chi_0)_{\bar{q}r}\Psi^r_{ij}) + T_{pi}{}^r R^p{}_r{}^k{}_j 
			+ \nabla_i \nabla_j T^{p k}{}_p\\
   & \ \ \ \ \ \ + \nabla_i \nabla^p T_{pj}{}^k 
			 + \nabla^\qbar \hat{R}_{\qbar i}{}^k{}_j \\
            \geq& -C S - C|\overline \nabla \Psi| - C|\nabla \Psi|  - C,
\end{align*}
where the last inequality follows from the bounds obtained in Equations (3.5)-(3.7) of \cite{sw13} and the assumption that $t\ge \varepsilon$. In order to bound the last term appearing on the right-hand side of \eqref{DeltaS}, we apply the following identity to convert the second Chern Ricci terms into first Chern Ricci terms
\begin{align*}
    R^p{}_p{}^i{}_r &= R^p{}_r{}^i{}_p - \nabla^p T_{pr}{}^i\\
    &= R^i{}_r{}^p{}_p - \nabla^pT_{pi}{}^i - \nabla_rT^{pi}{}_p\\
    &= -\frac{1}{t}(\chi^{i\bar{q}}\chi_{\bar{q}r}-\chi^{i\bar{q}}(\chi_0)_{\bar{q}r})- \nabla^pT_{pi}{}^i - \nabla_rT^{pi}{}_p,
\end{align*}
where the last equality follows from \eqref{Ricci_identity}. From the lower bound from Equation (3.5) of \cite{sw13}, we have that
\begin{align*}
    (R^p{}_p{}^i{}_r \Psi^{rj}{}_k + R^p{}_p{}^j{}_r \Psi^{ir}{}_k-R^p{}_p{}^r{}_k\Psi^{ij}{}_r)\Psi_{ij}{}^k\ge -C(S^{3/2}+1).
\end{align*}
Putting this together into \eqref{DeltaS} with an application of Young's inequality, we arrive at
\begin{align*}
    \Delta S \ge \frac{1}{2}(|\nabla \Psi|^2+|\bar\nabla\Psi|^2)-CS^{{3/2}}-C.
\end{align*}
We would like to show a uniform bound on our quantity $S$ on $\overline{B_{r/2}}$. Choose a smooth cutoff function $\rho\leq 1$ which is identically $1$ on $\bar B_{r/2}$ and supported on $B_r$ with
\begin{align*}
|\nabla\rho|^2 \leq \frac{C}{r^2}, \ \ \ |\Delta \rho| \leq \frac{C}{r^2}.
\end{align*}
By our assumption in the theorem on the uniform equivalence of $\hat\chi$ and $\chi$, we may take $N$ sufficiently large so that
\begin{align*}
    \frac{N}{2}\le N-\tr_{\hat\chi}\chi\le N.
\end{align*}
We will apply the maximum principle to the quantity 
\begin{align}
    f=\rho^2 \frac{S}{N-\tr_{\hat\chi}\chi}+A\tr_{\hat\chi}\chi.
\end{align}
Suppose that the maximum of $f$ on $\overline{B_r}$ is attained at a point $x_0.$ Let us first consider the case where $x_0$ is an interior point of $B_r$. We may assume that at $x_0$, $S\ge 1$ since otherwise we would already have our desired uniform bound on $S$. Let us compute at $x_0,$
\begin{align*}
    \Delta \left(\frac{\rho^2S}{N-\tr_{\hat\chi}\chi}\right) &=  \Delta \rho^2 \frac{S}{N-\tr_{\hat\chi}\chi} + \rho^2 \Delta \frac{S}{N-\tr_{\hat\chi}\chi} + 2\Re \left\langle \nabla \rho^2 , \nabla \frac{S}{N-\tr_{\hat\chi}\chi}\right\rangle  \\
    &=  \Delta \rho^2 \frac{S}{N-\tr_{\hat\chi}\chi} + \frac{\rho^2}{N-\tr_{\hat\chi}\chi}\Delta S + 2\frac{\rho^2}{(N-\tr_{\hat\chi}\chi)^2}\Re\langle \nabla S, \nabla \tr_{\hat\chi}\chi\rangle \\
    & \ \ \ \ \ \ + \rho^2 S \frac{\Delta \tr_{\hat\chi}\chi}{N-\tr_{\hat\chi}\chi}+ 2\rho^2 S\frac{|\nabla \tr_{\hat\chi}\chi|^2}{(N-\tr_{\hat\chi}\chi)^3} + 2\Re \left\langle \nabla \rho^2 , \frac{\nabla S}{N-\tr_{\hat\chi}\chi}\right\rangle\\
    & \ \ \ \ \ \ + 2\Re \left\langle \nabla \rho^2 , \frac{S\nabla \tr_{\hat\chi}\chi}{(N-\tr_{\hat\chi}\chi)^2}\right\rangle
\end{align*}
We know that at a max of $f$, we have that
\begin{align*}
    0 &= \nabla\left(\frac{\rho^2 S}{N-\tr_{\hat\chi}\chi}+A\tr_{\hat\chi}\chi\right) \\
      &= \nabla\rho^2 \frac{S}{N-\tr_{\hat\chi}\chi} + \rho^2 \frac{\nabla S}{N-\tr_{\hat\chi}\chi}+\rho^2\frac{S \nabla \tr_{\hat\chi}\chi}{(N-\tr_{\hat\chi}\chi)^2}+A\nabla \tr_{\hat\chi}\chi.
\end{align*}
So we get that
\begin{align*}
    \Delta \left(\frac{\rho^2 S}{N-\tr_{\hat\chi}\chi}\right) &= \Delta \rho^2 \frac{S}{N-\tr_{\hat\chi}\chi} + \frac{\rho^2}{N-\tr_{\hat\chi}\chi}\Delta S \\
    & \ \ \ \ \ \ + \rho^2 S \frac{\Delta \tr_{\hat\chi}\chi}{N-\tr_{\hat\chi}\chi} + 2\Re \langle \nabla \rho^2 ,  \frac{\nabla S}{N-\tr_{\hat\chi}\chi}\rangle\\
    & \ \ \ \ \ \ - 2A\frac{|\nabla\tr_{\hat\chi}\chi|^2}{N-\tr_{\hat\chi}\chi}.
\end{align*}
By an application of Young's inequality, it can be shown that
\begin{align*}
    \Delta \left(\frac{\rho^2S}{N-\tr_{\hat\chi}\chi}\right) \ge -\frac{C}{r^2}S -\frac{CAS}{N} ,
\end{align*}
where we used the fact that
\begin{align*}
    |\nabla S|^2 \le 2S(|\overline\nabla \Psi|^2+|\nabla\Psi|^2)
\end{align*}
and
\begin{align*}
    \Delta \tr_{\hat\chi}\chi \ge \frac{1}{2}S - C.
\end{align*}
The above inequality follows from the fact that $\chi$ solves the continuity equation, that torsion terms of $\chi$ can be written in terms of torsion terms of $\hat\chi$ as per Equation \eqref{Prelim:Torsion}, and applying Young's inequality to the last term in \eqref{Prelim:LaplacianTrace}.
Together this gives us that
\begin{align*}
    \Delta f \ge -\frac{C}{r^2}S-\frac{CAS}{N}+\frac{AS}{2}-CA.
\end{align*}
Choose $N$ big enough so that $-\frac{CAS}{N}+\frac{AS}{2}\ge \frac{AS}{4}$. Then pick $A = 4(\frac{C}{r^2} + 1)$.
At a maximum of $f$, we have that

\begin{align*}
    0 &\geq \Delta \left(\rho^2 \frac{S}{K-\tr_{\hat\chi}\chi} + A\tr_{\hat\chi}\chi\right) \\
      &\geq \left(\frac{A}{4}- \frac{C}{r^2}\right) S - CA \\
      &= S -CA
\end{align*}
This implies that $S(x_0) \leq CA$, therefore 
\begin{align*}
    \|S\|_{\overline {B_{r/2}}} \leq \frac{CNA}{\rho^2}\le \frac{C}{r^2}.
\end{align*}
This bound is independent of $T$. We are left to deal with the case where $x_0 \in \partial B_r$. In this case, $\rho=0$ and so we have $f(x_0)\le A\tr_{\hat\chi}\chi(x_0)\le \frac{C}{r^2}$ which implies
\begin{align*}
    \rho^2 \frac{S}{K-\tr_{\hat\chi}\chi} + A \tr_{\hat\chi}\chi 
    &\leq \frac{C}{r^2}.
\end{align*}
This gives us the desired bound 
\begin{align*}
    \|S\|_{\overline{B_{r/2}}} &\leq \frac{CN}{r^2},
\end{align*}
which is independent of $T$.
\end{proof} 

\subsection{Higher order estimates}\label{C3:Ck}

In this subsection we prove Corollary \ref{Intro:CorCk} by a standard bootstrapping argument. For simplicity we will only present the proof for solutions to equation \eqref{Intro:ContinuityEquation}.

\begin{proof}

Modulo the image of $i\ddb$, Equation \eqref{Intro:ContinuityEquation} says $[\omega - \omega_0] = -tc_1(X)$; therefore $\omega = \omega_0 - t\Ric\omega_0 + i\ddb\varphi$ for some $\varphi$ with $\varphi(0) = 0$. Thus
\begin{align*}
0
& = \omega_0 - t\Ric\omega - \omega \\
& = -t(\Ric\omega - \Ric\omega_0) - i\ddb\varphi \\
& = i\ddb\left(t\log\frac{\omega^n}{\omega_0^n} - \varphi\right).
\end{align*}
By normalizing $\varphi$ appropriately we may assume that $\varphi$ solves the following scalar equation
\begin{equation}\label{C3:Ck:ScalarEquation}
t\log\frac{\omega^n}{\omega_0^n} - \varphi = 0.
\end{equation}
Differentiating \eqref{C3:Ck:ScalarEquation} yields a linear equation
\begin{equation}\label{C3:Ck:LinearEquation}
t\Delta\varphi_\ell - \varphi_\ell = f
\end{equation}
for $\varphi_\ell = \partial_\ell\varphi$ where
\begin{equation*}
f = t^2g^{j\bar k}\partial_\ell R^0_{\bar kj}.
\end{equation*}
By the assumption $C^{-1}\omega_0\leq\omega\leq C\omega_0$ we may apply Theorem \ref{Intro:ThmC3} to the equation \eqref{Intro:ContinuityEquation}. Consequently the equation \eqref{C3:Ck:LinearEquation} has $C^{0,\alpha}$ coefficients. Schauder estimates imply that $\varphi_\ell$ belongs to $C^{2,\alpha}$, so $\varphi$ belongs to $C^{3,\alpha}$. But then \eqref{C3:Ck:LinearEquation} has $C^{1,\alpha}$ coefficients so we may reapply the argument at a higher degree of regularity. By continuing this process we obtain $C^k$ estimates of all orders for $\varphi$, which completes the proof of Corollary \ref{Intro:CorCk}.
\end{proof}

\section{Blow-up of the scalar curvature at a finite time singularity}\label{Blowup}

\noindent In this section we prove that the Chern scalar curvature of a solution to the continuity equation must blow up at a finite-time singularity.

\begin{proof}[Proof of Theorem \ref{Intro:ThmBlowup}]
For a contradiction, assume that $T < \infty$ and $R \leq C$ and that $\omega$ solves \eqref{Intro:ContinuityEquation}; a similar proof applies for \eqref{Intro:NormalizedContinuityEquation}. Contracting \eqref{Intro:ContinuityEquation} with $\omega$ yields
\begin{equation}\label{Blowup:InverseTraceEstimate}
\tr_\omega\omega_0 = n + tR \leq C.
\end{equation}
By the same argument as in subsection \ref{C3:Ck} we may assume that $\omega = \omega_0 - t\Ric\omega_0 + i\ddb\varphi$ where $\varphi$ solves the following scalar equation
\begin{equation}\label{Blowup:ScalarEquation}
t\log\frac{\omega^n}{\omega_0^n} - \varphi = 0
\end{equation}
with $\varphi(0) = 0$.

A maximum principle argument gives $\varphi\leq C$ and hence $\omega^n/\omega_0^n\leq e^{C/t}$. In combination with \eqref{Blowup:InverseTraceEstimate} this gives
\begin{equation}
\tr_{\omega_0}\omega \leq C\frac{\left(\tr_\omega\omega_0\right)^{n-1}}{\omega_0^n/\omega^n} \leq C\label{TraceFlip}
\end{equation}
for $t\geq T/2$. The estimate \eqref{Blowup:InverseTraceEstimate} together with the above estimate imply
\begin{equation*}
C^{-1}\omega_0\leq\omega\leq C\omega_0.
\end{equation*}
so we may apply Theorem \ref{Intro:ThmC3} and Corollary \ref{Intro:CorCk} to obtain $C^k$ estimates of all orders. By the Arzela-Ascoli theorem Equation \eqref{Intro:ContinuityEquation} can also be solved at $t = T$, and by the ellipticity of \eqref{Blowup:ScalarEquation} the equation can furthermore be solved at $t = T + \varepsilon$. This contradicts the maximality of $T$, proving the desired result.
\end{proof}

\section{Oeljeklaus-Toma manifolds}\label{OT}

\noindent In this section we will establish estimates which together prove Theorem \ref{Intro:ThmOT}. We start by reviewing some properties of Oeljeklaus-Toma manifolds, referring to \cite{ot05} for the details of their construction.

An Oeljeklaus-Toma manifold of complex dimension $n$ is a certain quotient of $\mathbb H^{n-1}\times\mathbb C$, on which we use the coordinates $z_1,\dotsc,z_{n-1}\in\mathbb{H}^{n-1}$ and $w\in\mathbb{C}$ and define $y_j = \Im z_j > 0$.

The $(1,1)$-forms
\begin{equation*}
\alpha = \frac{1}{4y_j^2}idz^j\wedge d\bar z^j
\end{equation*}
\begin{equation*}
\beta = y_1\cdots y_{n-1}idw\wedge d\bar w
\end{equation*}
\begin{equation*}
\gamma = \frac{1}{4y_jy_k}idz^j\wedge d\bar z^k
\end{equation*}
are well-defined. Setting
\begin{equation*}
\Omega = \alpha^{n-1}\wedge\beta
\end{equation*}
it follows that
\begin{equation*}
i\ddb\log\Omega = \alpha
\end{equation*}
and this shows that $-\alpha\in c_1(X)$.

A Hermitian metric $\hat\omega_0$ is called strongly flat along the leaves \cite{at21} if there is a $c > 0$ so that the following condition is satisfied:
\begin{equation}\label{OT:StronglyFlatAlongTheLeaves1}
\alpha^{n-1}\wedge\hat\omega_0 = c\Omega.
\end{equation}
Note that it is always possible to satisfy this property after a conformal change \cite[Lemma 2.2]{zheng15}:
\begin{equation*}
\hat\omega_0 \mapsto \frac{c\Omega}{\alpha^{n-1}\wedge\hat\omega_0}\hat\omega_0.
\end{equation*}
The property \eqref{OT:StronglyFlatAlongTheLeaves1} is equivalent to
\begin{equation}\label{OT:StronglyFlatAlongTheLeaves2}
\hat g^0_{\bar ww} = cy_1\cdots y_{n-1}.
\end{equation}
In particular, the Oeljeklaus-Toma metric
\begin{align}\label{OT:OT}
\omega_{OT} = \alpha + \beta + \gamma
\end{align}
satisfies \eqref{OT:StronglyFlatAlongTheLeaves1} and
\begin{align*}
\Ric(\omega_{OT}) = -\alpha.
\end{align*}
We remark that $\hat\omega = \omega_{OT} - t\Ric\omega_{OT}$ also satisfies $\Ric\hat\omega = -\alpha$, and is therefore the explicit solution of the normalized continuity equation \eqref{Intro:NormalizedContinuityEquation} starting at $\omega_{OT}$.

In the following subsections we will prove a priori estimates for a solution $\omega$ of \eqref{Intro:NormalizedContinuityEquation} with $\omega_0 = \hat\omega_0 + i\ddb\varphi_0$ and $\hat\omega_0$ satisfying \eqref{OT:StronglyFlatAlongTheLeaves1}. Modulo the image of $i\ddb$, \eqref{Intro:NormalizedContinuityEquation} says $(t+1)[\omega] = [\omega_0] - tc_1(X)$; therefore $\omega = \hat\omega + i\ddb\varphi$ for some $\varphi$ with $\varphi(0) = \varphi_0$ where
\begin{equation}\label{OT:OmegaHat}
\hat\omega = (\hat\omega_0 + t\alpha)/(t+1).
\end{equation}

\subsection{$C^0$ estimate}\label{OT:C0}

We begin by proving a $C^0$ estimate under a suitable normalization of the potential, and this implies a determinant estimate.

\begin{proposition}\label{OT:C0:Prop}
Let $\omega = \hat\omega + i\partial\bar\partial\varphi$ solve the normalized continuity equation \eqref{Intro:NormalizedContinuityEquation} with $\hat\omega_0$ satisfying \eqref{OT:StronglyFlatAlongTheLeaves1}. If $\varphi$ is normalized so that
\begin{equation}\label{OT:C0:Normalization}
\int_Xe^\frac{(t+1)\varphi - \varphi_0}{t}\Omega = \frac{(t+1)^n}{cnt^{n-1}}\int_X\omega^n
\end{equation}
then
\begin{equation}\label{OT:C0:PhiEstimate}
|\varphi| \leq C/(t+1)
\end{equation}
and
\begin{equation}\label{OT:C0:DeterminantEstimate}
e^{-C/t} \leq \frac{\omega^n}{\hat\omega^n} \leq e^{C/t}.
\end{equation}
\end{proposition}

\begin{proof}
Equation \ref{Intro:NormalizedContinuityEquation} takes the form
\begin{align*}
0
& = (t+1)\omega - (\omega_0 - t\Ric\omega) \\
& = (t+1)(\hat\omega + i\ddb\varphi) - (\omega_0 - t\Ric\omega) \nonumber \\
& = (\omega_0 + t\alpha + i\ddb((t+1)\varphi - \varphi_0)) - (\omega_0 - t\Ric\omega) \nonumber \\
& = i\ddb\left((t+1)\varphi - \varphi_0 - t\log\frac{\omega^n}{\Omega}\right).
\end{align*}
The normalization \eqref{OT:C0:Normalization} implies that $\varphi$ solves the following scalar equation
\begin{equation}\label{OT:C0:ScalarEquation}
(t+1)\varphi - \varphi_0 - t\log\left(\frac{(t+1)^n}{cnt^{n-1}}\frac{\omega^n}{\Omega}\right) = 0.
\end{equation}
We now use the assumption \eqref{OT:StronglyFlatAlongTheLeaves1} and the fact that $\alpha^n = 0$ to compute
\begin{align}\label{OT:C0:TimeTerm}
\frac{(t+1)^n}{cnt^{n-1}}\frac{\hat\omega^n}{\Omega}
& = \frac{1}{cnt^{n-1}}\frac{(\hat\omega_0 + t\alpha)^n}{\Omega} \\
& = \frac{1}{cnt^{n-1}}\frac{n\hat\omega_0\wedge(t\alpha)^{n-1} + \frac{n(n-1)}{2}\hat\omega_0^2\wedge(t\alpha)^{n-2} + \cdots + \hat\omega_0^n}{\Omega} \nonumber \\
& = 1 + \frac{f_1(x)}{t} + \cdots + \frac{f_{n-1}(x)}{t^{n-1}} \nonumber
\end{align}
where
\begin{align*}
f_{k-1} = \binom{n}{k}\frac{\hat\omega_0^k\wedge\alpha^{n-k}}{cn\Omega} \geq 0.
\end{align*}
For $t\geq 1$ we have
\begin{align*}
1 \leq 1 + \frac{f_1(x)}{t} + \cdots + \frac{f_{n-1}(x)}{t^{n-1}} \leq 1 + \frac{f(x)}{t}
\end{align*}
where $f = f_1 + \cdots + f_{n-1}$. Substituting this into \eqref{OT:C0:TimeTerm} and using $t\log(1 + f(x)/t)\leq f(x)$ we obtain
\begin{align*}
0 \leq t\log\left(\frac{(t+1)^n}{cnt^{n-1}}\frac{\hat\omega^n}{\Omega}\right) \leq f(x)
\end{align*}
and thus by \eqref{OT:C0:ScalarEquation}
\begin{equation}\label{OT:C0:ScalarInequality}
0 \leq (t+1)\varphi - \varphi_0 - t\log\frac{\omega^n}{\hat\omega^n} \leq f(x).
\end{equation}
Now at a maximum or minimum point of $\varphi$ we have $\omega\leq\hat\omega$ and $\omega\geq\hat\omega$ respectively, so \eqref{OT:C0:ScalarInequality} implies
\begin{equation}\label{OT:C0:ExplicitPhiEstimate}
-\|\varphi_0\|_{C^0} \leq (t+1)\varphi \leq \|\varphi_0\|_{C^0} + \|f\|_{C^0}
\end{equation}
which completes the proof of the $C^0$ estimate on $\varphi$. Substituting \eqref{OT:C0:ExplicitPhiEstimate} back into \eqref{OT:C0:ScalarInequality} gives
\begin{align*}
-2\|\varphi_0\|_{C^0} - \|f\|_{C^0} \leq t\log\frac{\omega^n}{\hat\omega^n} \leq 2\|\varphi_0\|_{C^0} + \|f\|_{C^0}
\end{align*}
which proves \eqref{OT:C0:DeterminantEstimate}.
\end{proof}

\subsection{$C^2$ estimate}\label{OT:C2}

In this subsection, we begin by proving a $C^2$ estimate in Proposition \ref{OT:C2:Prop1} by applying a maximum principle argument and controlling the torsion terms. Using this, we prove a refined $C^2$ estimate in Proposition \ref{OT:C2:Prop2} from which we can obtain Gromov-Hausdorff convergence of $(X,\omega)$ to $(\mathbb{T}^{n-1},d)$ following the proof in \cite{zheng15}, where $d$ is a flat Riemannian metric on the torus $\mathbb{T}^{n-1}.$

\begin{proposition}\label{OT:C2:Prop1}
Let $\omega = \hat\omega + i\partial\bar\partial\varphi$ solve the normalized continuity equation \ref{Intro:NormalizedContinuityEquation} starting at $\omega_0 = \hat\omega_0 + i\partial\bar\partial\varphi_0$ where $\hat\omega_0$ satisfies condition \eqref{OT:StronglyFlatAlongTheLeaves1}. Then there exists a uniform constant $C$ such that
\begin{equation}\label{OT:C2:UniformEquivalence}
C^{-1}\hat\omega\leq\omega\leq C\hat\omega.
\end{equation}
\end{proposition}

\begin{proof}
Following \cite{yau78} we will obtain the $C^2$ estimate by applying the maximum principle to a function with leading term $\log\tr_{\hat\omega}\omega$. We will use $R$ and $T$ to denote the curvature and torsion of the Chern connection $\nabla$ of $\hat\omega$, and $\Delta$ and $R(g)$ to denote the Laplacian and curvature of $\omega$. Equation \eqref{Intro:NormalizedContinuityEquation} gives
\begin{align}\label{OT:C2:Computation1}
-\hat g^{p\bar q}R_{\bar qp}(g)
& = \frac{(t+1)\tr_{\hat\omega}\omega - \tr_{\hat\omega}\omega_0}{t} \\
& \geq -\frac{1}{t}\tr_{\hat\omega}\omega_0 \nonumber \\
& \geq -C \nonumber
\end{align}
for $t\geq 1$, since \eqref{OT:OmegaHat} implies $\tr_{\hat\omega}\omega_0 \leq (t+1)\tr_{\hat\omega_0}\omega_0 \leq C(t+1)$.

Computing as in \cite[Lemma 3.2]{zheng15}, by keeping track of the factors of $t+1$ arising from \eqref{OT:OmegaHat}, the property \eqref{OT:StronglyFlatAlongTheLeaves2} implies that
\begin{gather}
g^{j\bar k}R_{\bar kj}{}^{p\bar q}g_{\bar qp} \geq -C\sqrt{t+1}\tr_{\hat\omega}\omega\tr_\omega\hat\omega \label{OT:C2:Computation2} \\
g^{j\bar k}R'_{\bar kj} - g^{j\bar k}R_{\bar kj} - \hat g^{p\bar q}g^{j\bar k}\hat g_{\bar sr}T^r_{pj}\bar T^s_{qk} \geq -C\tr_\omega\hat\omega. \label{OT:C2:Computation3}
\end{gather}
Note that without the property \eqref{OT:StronglyFlatAlongTheLeaves1}, the first estimate \eqref{OT:C2:Computation2} would have only a factor of $t+1$ on the right-hand side.

Substituting \eqref{OT:C2:Computation1}, \eqref{OT:C2:Computation2}, \eqref{OT:C2:Computation3} into \ref{Prelim:LaplacianLogTrace} and applying \eqref{Prelim:AubinYau} yields
\begin{align}\label{OT:C2:LaplacianLogTrace1}
\Delta\log\tr_{\hat\omega}\omega
& \geq \frac{1}{\tr_{\hat\omega}\omega}\bigg\{-C\sqrt{t+1}\tr_{\hat\omega}\omega\tr_\omega\hat\omega - C\tr_\omega\hat\omega - C + 2\Re g^{j\bar k}\frac{\partial_j\tr_{\hat\omega}\omega}{\tr_{\hat\omega}\omega}\bar T_k\bigg\} \nonumber \\
& \geq -C\sqrt{t+1}\tr_\omega\hat\omega - C + \frac{1}{\tr_{\hat\omega}\omega}2\Re g^{j\bar k}\frac{\partial_j\tr_{\hat\omega}\omega}{\tr_{\hat\omega}\omega}\bar T_k.
\end{align}
Here we have used $\tr_{\hat\omega}\omega\geq C^{-1}$, which follows from \eqref{OT:C0:DeterminantEstimate}.

We now work at a maximum point of
\begin{align*}
G = \log\tr_{\hat\omega}\omega - A\sqrt{t+1}\varphi + (t+1)\varphi^2
\end{align*}
where $A$ is a large constant to be chosen later. At this point
\begin{align*}
\frac{\partial_j\tr_{\hat\omega}\omega}{\tr_{\hat\omega}\omega} = A\sqrt{t+1}\partial_j\varphi - 2(t+1)\varphi\partial_j\varphi
\end{align*}
so
\begin{align*}
\frac{1}{\tr_{\hat\omega}\omega}2\Re g^{j\bar k}\frac{\partial_j\tr_{\hat\omega}\omega}{\tr_{\hat\omega}\omega}\bar T_k
& = \frac{1}{\tr_{\hat\omega}\omega}2\Re g^{j\bar k}\bigg(A\sqrt{t+1}\partial_j\varphi - 2(t+1)\varphi\partial_j\varphi\bigg)\bar T_k \\
& \geq -2(t+1)g^{j\bar k}\partial_j\varphi\partial_{\bar k}\varphi - \frac{CA^2}{(\tr_{\hat\omega}\omega)^2}\tr_\omega\hat\omega
\end{align*}
where we have used the fact that $g^{j\bar k}T_j\bar T_k \leq C\tr_\omega\hat\omega$. We also have
\begin{align*}
\Delta\varphi &= n - \tr_\omega\hat\omega \\
\Delta\varphi^2 = 2g^{j\bar k}&\partial_j\varphi\partial_{\bar k}\varphi + 2\varphi(n - \tr_\omega\hat\omega) 
\end{align*}
so by the $C^0$ estimate
\begin{align*}
\frac{1}{\tr_{\hat\omega}\omega}2\Re g^{j\bar k}\frac{\partial_j\tr_{\hat\omega}\omega}{\tr_{\hat\omega}\omega}\bar T_k + \Delta(-A\sqrt{t+1}\varphi + (t+1)\varphi^2) \quad \quad \quad \quad \quad \quad \\
\quad \quad \quad \quad \quad \quad \geq -\frac{CA^2}{(\tr_{\hat\omega}\omega)^2}\tr_\omega\hat\omega + (A\sqrt{t+1} - C)\tr_\omega\hat\omega - CA\sqrt{t+1}. \nonumber
\end{align*}
Since $\Delta G \leq 0$ at the maximum point, we substitute the above inequality into \eqref{OT:C2:LaplacianLogTrace1} to obtain
\begin{align*}
0 \geq -\frac{CA^2}{(\tr_{\hat\omega}\omega)^2}\tr_\omega\hat\omega + (A - C)\sqrt{t+1}\tr_\omega\hat\omega - CA\sqrt{t+1}.
\end{align*}
Choosing $A \gg C$ we can make the coefficient of the second term positive, so that
\begin{equation}\label{OT:C2:MainInequality1}
0 \geq -\frac{C}{(\tr_{\hat\omega}\omega)^2}\tr_\omega\hat\omega + (\sqrt{t+1}+1)\tr_\omega\hat\omega - C\sqrt{t+1}.
\end{equation}
We now consider two cases. Firstly, if
\begin{equation}\label{OT:C2:Case1}
-\frac{C}{(\tr_{\hat\omega}\omega)^2} \leq -1
\end{equation}
then we immediately obtain $\tr_{\hat\omega}\omega\le C$ at the maximum point which implies $$\tr_\omega\hat\omega\le C$$ by the determinant estimate \eqref{OT:C0:DeterminantEstimate} and the identity \eqref{TraceFlip}. On the other hand, if \eqref{OT:C2:Case1} fails then \eqref{OT:C2:MainInequality1} gives
\begin{align*}
\tr_\omega\hat\omega \le C
\end{align*}
which implies $\tr_{\hat\omega}\omega\le C$ at the maximum point. Thus in both cases, using the fact that we have a $C^0$ estimate on $\varphi$, we obtain uniform equivalence of the metrics as stated in desired.
\end{proof}

\noindent Next we sharpen \eqref{OT:C2:UniformEquivalence}.

\begin{proposition}\label{OT:C2:Prop2}
Let $\omega$ be a solution to the continuity equation \eqref{Intro:NormalizedContinuityEquation}. Then
\begin{equation}\label{OT:C2:TightUniformEquivalence}
(1-\varepsilon(t))\hat\omega\leq\omega\leq (1+\varepsilon(t))\hat\omega
\end{equation}
where $\varepsilon(t)\to 0$ as $t\to\infty$.
\end{proposition}

\begin{proof}
If $\sigma_k$ denotes the elementary symmetric polynomial of degree $k$ in the variables $\lambda_j$, then Maclaurin's inequality
\begin{align*}
\bigg(\frac{\sigma_{n-1}}{n}\bigg)^\frac{1}{n-1} \leq \frac{\sigma_1}{n}
\end{align*}
implies
\begin{equation}\label{OT:C2:TraceInverseInequality}
\tr_\omega\hat\omega \leq \frac{1}{n^{n-2}}\frac{\left(\tr_{\hat\omega}\omega\right)^{n-1}}{\omega^n/\hat\omega^n}
\end{equation}
and similarly
\begin{equation}\label{OT:C2:TraceInequality}
\tr_{\hat\omega}\omega \leq \frac{1}{n^{n-2}}\frac{\left(\tr_\omega\hat\omega\right)^{n-1}}{\hat\omega^n/\omega^n}.
\end{equation}
By \eqref{Prelim:LaplacianTrace} and our earlier computations \eqref{OT:C2:Computation1}, \eqref{OT:C2:Computation2}, \eqref{OT:C2:Computation3} we have
\begin{align*}
\Delta\tr_{\hat\omega}\omega \geq -C\sqrt{t+1}
\end{align*}
and so it follows that
\begin{align*}
\Delta(\tr_{\hat\omega}\omega - A\varphi) \geq A(\tr_\omega\hat\omega - n) - C\sqrt{t+1}.
\end{align*}
At a maximum point of $\tr_{\hat\omega}\omega - A\varphi$ we have
\begin{align*}
\tr_\omega\hat\omega \leq n + \frac{C\sqrt{t+1}}{A}
\end{align*}
so by \eqref{OT:C2:TraceInequality} and \eqref{OT:C0:DeterminantEstimate}
\begin{align*}
\tr_{\hat\omega}\omega \leq \frac{e^{C/t}}{n^{n-2}}\left(n + \frac{C\sqrt{t+1}}{A}\right)^{n-1}
\end{align*}
at this point; at an arbitrary point
\begin{align*}
\tr_{\hat\omega}\omega
& \leq \frac{e^{C/t}}{n^{n-2}}\left(n + \frac{C\sqrt{t+1}}{A}\right)^{n-1} + 2A\|\varphi\|_{C^0} \\
& \leq \frac{e^{C/t}}{n^{n-2}}\left(n + \frac{C\sqrt{t+1}}{A}\right)^{n-1} + \frac{CA}{t+1} \nonumber
\end{align*}
by \eqref{OT:C0:PhiEstimate}. After choosing $A = (t+1)^\frac34$ this becomes
\begin{align*}
\tr_{\hat\omega}\omega \leq \frac{e^{C/t}}{n^{n-2}}\left(n + \frac{C}{(t+1)^\frac14}\right)^{n-1} + \frac{C}{(t+1)^\frac14}
\end{align*}
which implies $\tr_{\hat\omega}\omega\leq n + \varepsilon(t)$. Applying \eqref{OT:C2:TraceInverseInequality} gives $\tr_\omega\hat\omega\le n+\varepsilon(t)$. Finally, applying \cite[Lemma 2.6]{twy14} gives \eqref{OT:C2:TightUniformEquivalence}. 
\end{proof}

\noindent From Proposition \ref{OT:C2:Prop2} we may deduce that $(X,\omega)$ converges in the Gromov-Hausdorff sense to $\mathbb T^{n-1}$ in the same way as in the proof of Theorem 1.1 in \cite{zheng15}.

\subsection{Higher order estimates}\label{OT:HO}

In this subsection we prove higher order estimates using our $C^2$ estimate. The strategy is as follows: we first use a rescaling argument on the metric so that we can apply Theorem \ref{C3} to obtain the third order estimate, and once we have the third order estimate, we can bootstrap the scalar equation of the continuity equation to obtain all higher order estimates. In this section, we abuse the notation $\varphi$, $\omega$ and $\hat \omega$ for the objects on $X$ and their pullback on $\mathbb H^{n-1} \times \mathbb C$, which should be clear from the context.

\begin{definition}\label{weaklyparallel}
    We say that a real $(1,1)$-form $\Theta$ on $\mathbb H^{n-1} \times \mathbb C$ is weakly parallel along the leaves if $\frac{\partial}{\partial w} \Theta_{\overline w w} = 0$.
\end{definition}

\noindent We remark that this condition is well-defined since the coordinates we are working in are global. The metric $\omega_{OT}$ defined in \eqref{OT:OT} is an example of a $(1,1)$ form that is weakly parallel along the leaves.

\begin{proposition}\label{OT:HO:Prop}
    Given $C^{-1} \hat \omega \leq \omega \leq C \hat\omega$ on $X \times [0,\infty)$ and $\sqrt{-1}\ddb \varphi_0$ is weakly parallel along the leaves, we have 
    \begin{equation*}
    \|\varphi\|_{C^k(X,\omega_E)}\le C_k
    \end{equation*}
     where $C_k$ is a uniform constant that may depend on $k,C$.
\end{proposition}

\noindent Notice Theorem \ref{Intro:ThmC3} doesn't apply directly since the reference metric $\hat\omega$ is moving and degenerates when $t\to\infty$. To deal with this, we adapt the following rescaling argument used by Angella-Tosatti in \cite{at21}. 

\begin{proof}
Let $\lambda_t: \mathbb H^{n-1}\times\mathbb C \rightarrow\mathbb H^{n-1}\times\mathbb C$ given by $$\lambda_t (z_1,\ldots,z_{n-1},w)=(z_1,\ldots,z_{n-1},\sqrt{t+1}w),$$ so that we are stretching in the leaf directions. It follows that
\begin{align*}
\lambda_t^*\alpha = \alpha,\  \lambda_t^*\beta = (t+1)\beta, \ \lambda_t^*\gamma=\gamma.
\end{align*}
Recall $\hat \omega(t) = \frac{\omega_{OT} + t \alpha }{t+1} = \alpha+ \frac{\beta}{t+1} + \frac{\gamma}{t+1}$. It follows that 
\begin{align}
     \lambda^*_t \omega_{OT} &=  \alpha + (t+1) \beta + \gamma, \nonumber\\
     \lambda^*_t \hat \omega(t) &=  \alpha+ \beta+ \frac{\gamma}{t+1}. \label{OT:HO:stretchedrefmetric}
\end{align}

Since $\omega$ is a solution of \eqref{Intro:ContinuityEquation}, it follows that $\lambda_t^*\omega(t)$ satisfies
\begin{equation}
    \frac{(t+1)\lambda_t^*\omega(t) - \lambda_t^*\omega(0)}{t} = -\Ric(\lambda_t^*\omega(t)) \label{OT:HO:pullbackeqn}
\end{equation}
for any $t \in [0,\infty)$.

Next we want to apply the Calabi estimate to \eqref{OT:HO:pullbackeqn}. To do that, we need to establish the uniform equivalence between $\lambda_t^*\omega(t)$ and a fixed background metric (i.e. the Euclidean metric $\omega_E$) on compact subsets of $\mathbb H^{n-1}\times\mathbb C$.

By \eqref{OT:HO:stretchedrefmetric}, we have for any $t \in [0,\infty)$,
\begin{align*}
    C^{-1} \omega_E \leq \alpha + \beta \leq \lambda_t^* \hat \omega(t) \leq \alpha +\beta + \gamma \leq C \omega_E
\end{align*}
since both $ \alpha + \beta$ and $\alpha +\beta + \gamma $ are locally uniformly equivalent to $\omega_E$. It follows from the $C^2$ estimate of Proposition \ref{OT:C2:Prop1} and the above inequality that on any compact subset $K$ of $\mathbb H^{n-1}\times\mathbb C$ for all $t \ge 0$,
\begin{align*}
    C_K^{-1}\omega_E\le  C^{-1}\lambda_t^*\hat \omega(t) \le \lambda_t^*\omega(t) \le  C\lambda_t^*\hat \omega(t)  \le C_K \omega_E.
\end{align*}

\noindent Now we are in position to apply Theorem \ref{Intro:ThmC3} to the equation \eqref{OT:HO:pullbackeqn} with $\hat \chi = \omega_E$ and $\chi_0 = \lambda_t^*\omega_0$. Notice $\lambda_t^* \hat \omega(t)$ is smoothly uniformly bounded and it lies in the same $\ddb$-class as $\lambda_t^* \omega(t)$.  It follows the torsion tensor of $\lambda_t^* \omega(t)$ is smoothly uniformly bounded. Moreover, since we assume $\sqrt{-1}\ddb\varphi_0$ is weakly parallel along the leaves, we have $\frac{\|\nabla_E \lambda_t^* \omega_0\|_{\omega_E}}{t}$ is uniformly bounded on $[1,\infty)$. Tracing through the proof of Theorem \ref{Intro:ThmC3}, a bound on the first derivative of the initial metric by a linear factor of $t$ is sufficient for the estimate to go through. Thus, all the assumptions for Theorem \ref{Intro:ThmC3} are fulfilled. So we have
\begin{align*}
    \|\lambda_t^* \omega(t) \|_{C^1(K,\omega_E)}\le C_{K}
\end{align*}
for $t\geq 1$. Next we transfer the estimate from compact subsets on $\mathbb H^{n-1} \times \mathbb C$ to the OT manifold $X$. Let $\hat K$ be a fixed compact set which contains $\lambda^{-1}_t (K)$ for all $t$. For example we could take $\hat K$ to be the closure of $\bigcup_t \lambda^{-1}_t (K)$, which is bounded by definition of $\lambda_t$. We get
\begin{align*}  \|\omega(t)\|_{C^1(K,\omega_E)}&=\|\lambda_t^*\omega(t)\|_{C^1(\lambda_t^{-1}K,\lambda_t^*\omega_E)}\\
&\le \|\lambda_t^*\omega(t)\|_{C^1(\lambda_t^{-1}K,\omega_E)}\\
&\le \|\lambda_t^*\omega(t)\|_{C^1(\hat K,\omega_E)}\\
&\le C_{\hat K}.
\end{align*}
for all $t\ge 1$. The first inequality follows from $\lambda_t^* \omega_E \ge \omega_E$. Note that the constant dependence on $\hat K$ can be reduced to $K$ since $\hat{K}$ only depends on $K$ but not on $t$. 

Now we take $K$ to be the closure of the fundamental domain; thus, we have established the third order estimate:
\begin{align*}
    \|\omega(t)\|_{C^1(X,\omega_{OT})} \le C.
\end{align*}
Finally we may apply Corollary \ref{Intro:CorCk} to obtain $C^k$ estimates of all orders for $\varphi$, which completes the proof of Proposition \ref{OT:HO:Prop}.
\end{proof}

\noindent Proposition \ref{OT:C2:Prop2} and Proposition \ref{OT:HO:Prop} together imply that $\omega$ converges smoothly to $-\alpha$ as $t\to\infty$.

\subsection{Ricci curvature bounds}\label{OT:Ricci}

In this subsection we prove the final claim of Theorem \ref{Intro:ThmOT}. The lower bound on the Ricci curvature holds immediately from the definition of the continuity equation, while the upper bound on the Ricci curvature relies on our $C^2$ estimate.

\begin{proposition}\label{OT:Ricci:Prop}
We have
\begin{align*}\label{OT:Ricci:Estimate}
    -C\omega\leq\Ric\omega\leq C\omega.
\end{align*}
\end{proposition}
\begin{proof}
From Equation \eqref{Intro:NormalizedContinuityEquation} we have
\begin{align*}
    \Ric\omega = -\frac{(t+1)\omega - \omega_0}{t},
\end{align*}
which gives us that
\begin{align*}
    -\left(1+\frac{1}{t}\right)\omega\leq\Ric\omega\leq\frac{1}{t}\omega_0.
\end{align*}
Since $\omega_0\leq C\hat\omega_0 \leq C(t+1)\hat\omega\leq C(t+1)\omega$ by Proposition \ref{OT:C2:Prop1} and $t\ge 1$, we obtain the desired Ricci curvature bounds.
\end{proof}

\section*{Acknowledgments}

\noindent The authors are very grateful to Valentino Tosatti for numerous helpful correspondences. We would also like to thank Nikita Klemyatin for his useful suggestions. Finally we would like to thank D. H. Phong and Ben Weinkove for their continued support and feedback.

\bibliography{references}
\end{document}